\documentclass[11pt]{amsart}

\usepackage{paralist, amsmath, amsthm, amssymb, color}
\usepackage{tikz}
\tikzstyle{bv}=[circle,draw=black!90,fill=black!100,thick,inner
sep=2pt,minimum width=5pt]
\usepackage{hyperref}

\newtheorem{thm}{Theorem}
\numberwithin{equation}{section}
\numberwithin{thm}{section}
\newtheorem{theorem}[thm]{Theorem}

\newtheorem{lemma}[thm]{Lemma}
\newtheorem{corollary}[thm]{Corollary} 

\newtheorem{definition}[thm]{Definition}

\renewcommand{\Pr}{\mathbb{P}}

\newcommand{\pp}{p}
\newcommand{\mm}{m}

\newcommand{\om}{\omega}

\newcommand{\ga}{\gamma}
\newcommand{\de}{\delta}
\newcommand{\al}{\alpha}

\newcommand{\la}{\lambda}

\newcommand{\mC}{\mathcal{C}}
\newcommand{\mM}{\mathcal{M}}
\newcommand{\mT}{\mathcal{T}}

\title[Bijections for the factorizations of the long cycle]{Bijections and symmetries for the factorizations of the long cycle}

\author{Olivier Bernardi and Alejandro H. Morales}
\thanks{O. Bernardi aknowledges support from NSF grant DMS-1068626, ERC ExploreMaps and ANR A3.}
\date{\today}

\begin{document}
\setcounter{tocdepth}{2}

\begin{abstract}
We study the factorizations of the permutation $(1,2,\ldots,n)$ into $k$ factors of given cycle types. Using representation theory, Jackson obtained for each $k$ an elegant formula for counting these factorizations according to the number of cycles of each factor. In the cases $k=2,3$ Schaeffer and Vassilieva gave a combinatorial proof of Jackson's formula, and Morales and Vassilieva obtained more refined formulas exhibiting a surprising symmetry property. These counting results are indicative of a rich combinatorial theory which has remained elusive to this point, and it is the goal of this article to establish a series of bijections which unveil some of the combinatorial properties of the factorizations of $(1,2,\ldots,n)$ into $k$ factors for all $k$. We thereby obtain refinements of Jackson's formulas which extend the cases $k=2,3$ treated by Morales and Vassilieva. Our bijections are described in terms of ``constellations'', which are graphs embedded in surfaces encoding the transitive factorizations of permutations.\\ 
\end{abstract}

\maketitle

\section{Introduction}\label{sec:intro}
We consider the problem of enumerating the factorizations of the permutation $(1,2,\ldots,n)$ into $k$ factors according to the cycle type of each factor. In \cite{DMJ} Jackson established a remarkable \emph{counting formula} (analogous to the celebrated Harer-Zagier formula \cite{HZ}) characterizing the generating function of the factorizations of the long cycle according to the number of cycles of each factor. A combinatorial proof was subsequently given for the cases $k=2,3$ by Schaeffer and Vassilieva \cite{SV,SV2}. Building on these bijections, Morales and Vassilieva also established for $k=2,3$ a formula for the generating function of factorizations of $(1,2,\ldots,n)$ counted according to the cycle type of each factor \cite{MV1,MV2}. These formulas display a surprising \emph{symmetry property} which has remained unexplained so far.

In this article we explore the combinatorics of the factorizations of the permutation $(1,2,\ldots,n)$ through a series of bijections. All our bijections are described in terms of \emph{maps} and \emph{constellations} which are graphs embedded in surfaces encoding the transitive factorizations in the symmetric group (see Section~\ref{sec:def} for definitions). A summary of our bijections is illustrated in Figure~\ref{fig:summary}.
Our first bijection gives an encoding of the factorizations of the permutation $(1,2,\ldots,n)$ into \emph{tree-rooted $k$-constellations}. This encoding allows one to easily prove the case $k=2$ of Jackson's counting formula, as well as to establish the symmetry property for all $k\geq 2$. However for $k\geq 3$, the tree-rooted $k$-constellations are still uneasy to count and we give further bijections. Eventually, we show bijectively that proving Jackson's counting formula reduces to proving an intriguing probabilistic statement (see Theorem~\ref{thm:tree-puzzle}). In Section~\ref{sec:smallk} we prove this probabilistic statement in the cases $k=2,3,4$ (thereby proving Jackson's counting formula for these cases) but the cases $k>4$ shall be treated (along with similar probabilistic statements) in a separate paper \cite{BM:trees-from-sets}. 
Before describing our results further we need to review the literature.\\

\begin{figure} 
\begin{center} 
\includegraphics[width=\linewidth]{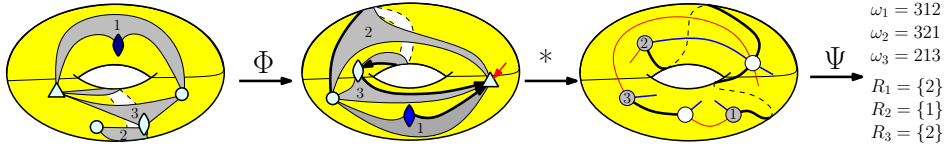} 
\caption{Summary of the bijections presented in this article. We start with the classical encoding of the factorizations of $(1,2,\ldots,n)$ by cacti (here $k=n=3$). We then establish a bijection $\Phi$ between  vertex-colored  cacti and tree-rooted constellations (Section~\ref{sec:tree-rooted}). We then  characterize the dual of tree-rooted constellations and obtain a correspondence with nebulas (Section \ref{sec:nebulas}). Lastly we establish a bijection $\Psi$ between nebulas and valid biddings.
} \label{fig:summary}
\end{center} 
\end{figure}

\noindent \textbf{Enumerative results about the factorizations of the long cycle.} Given $k$ partitions $\la^{(1)},\ldots,\la^{(k)}$ of $n$, it is a classical problem to determine the number $\kappa(\la^{(1)},\ldots,\la^{(k)})$ of factorizations $\pi_1\circ \pi_2\circ \cdots \circ \pi_k=(1,2,\ldots,n)$ such that the permutation $\pi_t$ has cycle type $\la^{(t)}$ for all $t\in\{1,\ldots,k\}$. 
By the general theory of group representations, the \emph{connection coefficients} $\kappa(\la^{(1)},\ldots,\la^{(k)})$ can be expressed in terms of the characters of the symmetric group, but this expression is not really explicit even for $k=2$. However, Jackson established in \cite{DMJ} a remarkable formula for the generating function of factorizations counted according to the number of cycles of the factors, namely,
\begin{equation} \label{eq:Jackson-GF}
\sum_{\pi_1\circ \cdots \circ \pi_k=(1,2,\ldots,n)}\prod_{i=1}^k x_i^{\ell(\pi_i)}= \sum_{1\leq p_1,\ldots,p_k\leq n}\prod_{i=1}^k{x_i\choose p_i} n!^{k-1}M^{n-1}_{p_1-1,\ldots,p_k-1}
\end{equation}
where $\ell(\pi)$ is the number of cycles of the permutation $\pi$, and $M^n_{p_1,\ldots,p_k}$ is the coefficient of $x_1^{p_1}\cdots x_k^{p_k}$ in the polynomial $(\prod_{i=1}^k(1+x_i)-\prod_{i=1}^k x_i)^n$. 

Jackson's formula can equivalently be stated in terms of \emph{colored factorizations}. 
\begin{definition}\label{def:colored-factorization}
Given positives integers $p_1,\ldots,p_k$, a $(p_1,\ldots,p_k)$\emph{-colored factorization} of $(1,2,\ldots,n)$ is a tuple $(\pi_1,\ldots, \pi_k,\phi_1,\ldots,\phi_k)$, where $\pi_1,\ldots, \pi_k$ are permutations of $[n]:=\{1,\ldots,n\}$ such that $\pi_1\circ\cdots\circ\pi_k=(1,2,\ldots,n)$ and for all $t\in[k]$, $\phi_t$ is a \emph{surjective} mapping from $[n]$ to $[p_t]$ such that $\phi_t(a)=\phi_t(b)$ if $a,b$ are in the same cycle of $\pi_t$. In other words, the mapping $\phi_t$ can be seen as a coloring of the cycles of the permutation $\pi_t$ with colors in $[p_t]$ and we want \emph{all the colors} to be used.
\end{definition}
 It is easy to see that \eqref{eq:Jackson-GF} is equivalent to the following theorem.
\begin{thm}[Jackson's counting formula \cite{DMJ}] \label{thm:Jackson-counting}
The number $C^n_{p_1,\ldots,p_k}$ of $(p_1,\ldots,p_k)$-colored factorizations of the permutation $(1,2,\ldots,n)$ is equal to 
\begin{equation} 
n!^{k-1}M^{n-1}_{p_1-1,\ldots,p_k-1},
\end{equation}
where $M^n_{p_1,\ldots,p_k}=[x_1^{p_1}\cdots x_k^{p_k}](\prod_{i=1}^k(1+x_i)-\prod_{i=1}^k x_i)^n$ is the cardinality of the set $\mM^n_{p_1,\ldots,p_k}$ of $n$-tuples $(R_1,\ldots,R_n)$ of strict subsets $R_t$ of $[k]$ such that each integer $t\in[k]$ appears in exactly $p_t$ of the subsets $R_1,\ldots,R_n$.
\end{thm}
The original proof of Theorem~\ref{thm:Jackson-counting} in \cite{DMJ} is based on the representation theory of the symmetric group. Bijections explaining the cases $k=2,3$ were subsequently given by Schaeffer and Vassilieva \cite{SV,SV2}. The case $k=2$ of Theorem~\ref{thm:Jackson-counting} is actually closely related to the celebrated Harer-Zagier formula \cite{HZ}, which was proved bijectively by Goulden and Nica \cite{Goulden:Harer-Zagier}. In Section~\ref{sec:tree-rooted} we shall give a bijection which extends the results in \cite{SV,SV2} to arbitrary $k$ (however, for a general $k$, this bijection does not directly imply Theorem~\ref{thm:Jackson-counting}).

We now consider a refined enumeration problem. Let $\ga^{(1)},\ldots,\ga^{(k)}$ be compositions of $n$, where $\ga^{(t)}=(\ga^{(t)}_{1},\ga^{(t)}_{2},\ldots,\ga^{(t)}_{p_t})$. We say that a $(p_1,\ldots,p_k)$-colored factorization $(\pi_1,\ldots, \pi_k,\phi_1,\ldots,\phi_k)$ has \emph{color-compositions} $(\ga^{(1)},\ldots,\ga^{(k)})$ if the permutation $\pi_t$ has $\ga^{(t)}_{i}$ elements colored $i$ (i.e. $\ga^{(t)}_{i}=|\phi_t^{-1}(i)|$) for all $t\in[k]$ and all $i\in[p_t]$. Let $c(\ga^{(1)},\ldots,\ga^{(k)})$ be the number of colored factorizations of color-compositions $(\ga^{(1)},\ldots,\ga^{(k)})$. In Section~\ref{sec:symmetry} we shall prove bijectively the following surprising symmetry property. 
\begin{thm}[Symmetry property] \label{thm:symmetry} Let $\ga^{(1)},\de^{(1)},\ldots,\ga^{(k)},\de^{(k)}$ be compositions of $n$.  
If for every $t\in[k]$ the length of the compositions $\ga^{(t)}$ and $\de^{(t)}$ are equal, then $c(\ga^{(1)},\ldots,\ga^{(k)})=c(\de^{(1)},\ldots,\de^{(k)})$.
\end{thm}
Given that there are ${n-1\choose \ell-1}$ compositions of $n$ with $\ell$ parts, the symmetry property together with Theorem~\ref{thm:Jackson-counting} gives the following refined formula.
\begin{corollary}
For any compositions $\ga^{(1)},\ldots,\ga^{(k)}$ of $n$, the number of colored factorizations of color-compositions $(\ga^{(1)},\ldots,\ga^{(k)})$ is
\begin{equation}\label{eq:MV}
c(\ga^{(1)},\ldots,\ga^{(k)})=\frac{n!^{k-1}M^{n-1}_{p_1-1,\ldots,p_k-1}}{\prod_{t=1}^k{n-1\choose \ell(\ga^{(t)})-1}},
\end{equation}
\end{corollary}
Equation \eqref{eq:MV} is precisely the result established by Morales and Vassilieva for the cases $k=2,3$ in \cite{MV1,MV2}. These refined results can actually be obtained by a representation theory approach but we have not found them explicitly in the literature for $k>3$.

As a final remark, observe that connection coefficients $\kappa(\la^{(1)},\ldots,\la^{(k)})$ are determined by the numbers $c(\la^{(1)},\ldots,\la^{(k)})$ of colored factorizations through a change of basis for symmetric functions:\\
$\displaystyle
\sum_{\la^{(1)},\la^{(2)},\ldots,\la^{(k)} }\kappa(\la^{(1)},\ldots,\la^{(k)})\prod_{t=1}^k\pp_{\la^{(t)}}(x_{t,1},x_{t,2},\ldots)\\
\indent \hspace{3cm} = \sum_{\la^{(1)},\la^{(2)},\ldots,\la^{(k)} }c(\la^{(1)},\ldots,\la^{(k)})\prod_{t=1}^k \mm_{\la^{(t)}}(x_{t,1},x_{t,2},\ldots), $\\
where the sums are over $k$-tuples of partitions of $n$ and $\pp_\la$, $\mm_\la$ denote respectively the \emph{power sum} and \emph{monomial} symmetric functions.\\

\noindent \textbf{A probabilistic puzzle.}
We now describe a probabilistic puzzle associated to the set $\mM^n_{p_1,\ldots,p_k}$ appearing in Theorem~\ref{thm:Jackson-counting}.
\begin{definition} \label{def:alpha}
For an integer $t$ in $[k]$ and a subset $R\subsetneq [k]$ we define the integer $\al(t,R)\in [k]$ by setting 
\begin{compactitem}
\item $\al(t,R)=t-1$ modulo $k$ if $t\in R$, 
\item $\al(t)=t+r$ modulo $k$ if $t\notin R$, $t+1,\ldots,t+r\in R$ and $t+r+1\notin R$.
\end{compactitem}
Furthermore, for integers $i_1,\ldots,i_{k-1}$ in $[n]$ (repetitions allowed) and subsets $R_1,\ldots,R_n\subsetneq [k]$, we define $\al((i_1,\ldots,i_{k-1}),(R_1,\ldots,R_n))$ to be the graph with vertex set $[k]$ and edge set $E=\{e_1,\ldots,e_{k-1}\}$, where $e_t$ is the edge $\{t,\al(t,R_{i_t})\}$.
\end{definition}

In Section \eqref{sec:biddings} we shall prove bijectively that Jackson's Theorem~\ref{thm:Jackson-counting} is equivalent to the following probabilistic statement.
\begin{thm}\label{thm:tree-puzzle} 
Fix an integer $k\geq 2$ and positive integers $n,p_1,\ldots,p_k$ and consider the uniform distribution on pairs $B=((i_1,\ldots,i_{k-1}),(R_1,\ldots,R_n))$ such that $i_1,\ldots,i_{k-1}$ are integers in $[n]$ (repetitions allowed) and $(R_1,\ldots,R_n)\in \mM^n_{p_1,\ldots,p_k}$. Then the probability that the graph $\al(B)$ is a tree is equal to the probability that the subset $R_1$ has cardinality $k-1$.
\end{thm}
In Section~\ref{sec:smallk}, we shall give a direct proof of Theorem~\ref{thm:tree-puzzle} for the cases $k=2,3,4$ (thereby establishing Theorem~\ref{thm:Jackson-counting} for these cases), and indicate how to handle a few more cases using a computer. In the forthcoming paper \cite{BM:trees-from-sets}, we shall give a direct proof of Theorem~\ref{thm:tree-puzzle} valid for all $k\geq 2$, and establish several other results of a similar flavor.\\

\noindent \textbf{Constellations, or the drawings of factorizations}. 
Our bijections are described in terms of \emph{$k$-constellations} which are certain \emph{maps} (embeddings of graphs in surfaces considered up to homeomorphism) with $k$ types of vertices and with faces colored white and black (black faces are also called \emph{hyperedges} and are colored in gray in our figures). 
We refer the reader to Section~\ref{sec:def} for precise definitions, and to Figure~\ref{fig:exsconstellations} for some examples.
To a $k$-constellation with its hyperedges labelled $1,\ldots,n$, one associates the permutations $\pi_1,\ldots,\pi_k$ of $[n]$ where the cycles of the permutation $\pi_t$ are in correspondence with the vertices of type $t$: the cycle associated to a vertex $v$ is given by the counterclockwise order of the black faces incident to $v$; see Figure~\ref{fig:exsconstellations}. Actually, any tuple of permutations $\pi_1,\ldots,\pi_k$ of $[n]$ acting transitively on $[n]$ is associated to a unique $k$-constellation, thus $k$-constellations give a canonical way of ``drawing'' such tuples. Moreover, the number of cycles of the product $\pi_1\circ \pi_2\cdots\pi_k$ corresponds to the number of white faces of the associated constellation. 
In particular, the tuples $(\pi_1,\ldots,\pi_k)$ of permutations of $[n]$ such that the product $\pi_1\circ\pi_2\cdots\pi_k$ is an $n$-cycle (a permutation with 1 cycle) correspond bijectively to constellations with a single white face (because the transitivity condition is automatically satisfied in this case). Constellations with a single white face are called \emph{cacti}.\\

\begin{figure} 
\begin{center} 
\includegraphics[width=\linewidth]{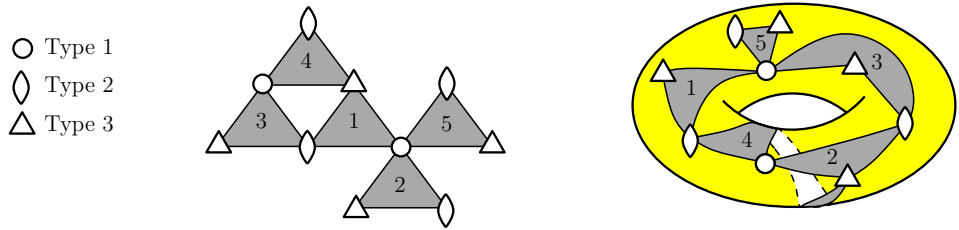} 
\caption{Two hyperedge-labelled $3$-constellations of size~5 (the shaded triangles represent the hyperedges). The 3-constellation on the left (which is embedded in the sphere) encodes the triple $(\pi_1,\pi_2,\pi_3)$, where $\pi_1=(1,2,5)(3,4)$, $\pi_2=(1,3)(2)(4)(5)$, $\pi_3=(1,4)(2)(3)(5)$, so that $\pi_1\pi_2\pi_3=(1,3,2,5)(4)$. The 3-cactus on the right (which is embedded in the torus) encodes the triple $(\pi_1,\pi_2,\pi_3)$, where $\pi_1=(1,3,5)(2,4)$, $\pi_2=(1,4)(2,3)(5)$, $\pi_3=(1)(2,4)(3)(5)$, so that $\pi_1\pi_2\pi_3=(1,2,3,4,5)$.
} \label{fig:exsconstellations}
\end{center} 
\end{figure}

\noindent \textbf{Outline.} 
In Section~\ref{sec:def}, we gather our definitions about maps, constellations and cacti. Theorems~\ref{thm:Jackson-counting} and~\ref{thm:symmetry} can then be stated in terms of vertex-colored cacti.

In Section~\ref{sec:tree-rooted}, we establish a bijection $\Phi$ between colored cacti and \emph{tree-rooted constellations} which are constellations with a marked spanning tree. Through this bijection the number of colors of the vertices of type $t$ in the cacti correspond to the number of vertices of type $t$ in the corresponding tree-rooted constellation.
This bijection is similar to the one used in \cite{OB:Harer-Zagier-non-orientable} to prove the Harer-Zagier formula \cite{HZ}.

In Section~\ref{sec:symmetry}, we establish a symmetry property for tree-rooted constellations which implies Theorem~\ref{thm:symmetry}. 

In Section~\ref{sec:nebulas}, we decompose the \emph{dual} of tree rooted maps and obtain certain decorated one-face maps that we call \emph{nebulas}. Proving Jackson's counting formula then reduces to counting nebulas. 

In Section~\ref{sec:biddings} we give a bijection $\Psi$ between nebulas and \emph{valid biddings}, where the definition of bidding is closely related to the set $\mM^n_{p_1,\ldots,p_k}$ whose cardinality appears in Jackson's formula. We then prove that Jackson's formula is equivalent to the probabilistic statement given in Theorem~\ref{thm:tree-puzzle}. 

Lastly, in Section~\ref{sec:smallk} we prove the cases $k=2,3,4$ of Theorem~\ref{thm:tree-puzzle}, thereby establishing Theorem~\ref{thm:Jackson-counting} for these cases.


\section{Definitions}\label{sec:def}
For a positive integer $n$, we denote by $[n]$ the set $\{1,2,\ldots,n\}$. A \emph{composition} of $n$ is a sequence of positive integers $\al=(\al_1,\al_2,\ldots,\al_\ell)$ such that $\al_1+\al_2+\cdots+\al_\ell=n$. The integers $\al_1,\ldots,\al_\ell$ are the \emph{parts} of $\al$ and the integer $\ell$ is the \emph{length} of $\al$. A \emph{partition} is a composition $\al=(\al_1,\al_2,\ldots,\al_\ell)$ such that $\al_1\geq \al_2\geq \cdots\geq \al_\ell$.\\

\noindent \textbf{Graphs and maps.}
Our \emph{graphs} are undirected and can have loops and multiple edges. A \emph{digraph}, or \emph{directed graph}, is a graph where every edge is oriented; oriented edges are called \emph{arcs}. An \emph{Eulerian tour} of a directed graph is a directed path starting and ending at the same vertex and taking every arc exactly once. 
An edge $e$ of a graph defines two \emph{half-edges} each of them incident to an endpoint of~$e$. A \emph{rotation system} for a graph $G$ is an assignment for each vertex $v$ of $G$ of a cyclic ordering for the half-edges incident to~$v$. 

We now review the connection between rotation systems and embeddings of graphs in surfaces. 
We call \emph{surface} a compact, connected, orientable, 2-dimensional manifold without boundary (such a surface is characterized by its \emph{genus} $g\geq 0$). A \emph{map} is a cellular embedding of a connected graph in an oriented surface considered up to orientation preserving homeomorphism\footnote{Maps can be considered on non-orientable surfaces but we will not consider such surfaces here.}. By \emph{cellular} we mean that the \emph{faces} (connected components of the complement of the graph) are simply connected. 
For a map, the angular section between two consecutive half-edges around a vertex is called a \emph{corner}. The \emph{degree} of a vertex or a face is the number of incident corners.
A map $M$ naturally defines a rotation system $\rho(M)$ of the underlying graph $G$ by taking the cyclic order of the half-edges incident to a vertex $v$ to be the clockwise order of these half-edges around $v$. The following classical result (see e.g. \cite{MT}) states the relation between maps and graphs with rotation systems.
\begin{lemma} \label{lem:embedding-map} 
For any connected graph $G$, the function $\rho$ is a bijection between the set of maps having underlying graph $G$ and the set of rotation systems of~$G$. 
\end{lemma}

\noindent \textbf{Constellations and cacti.} 
A \emph{$k$-constellation}, or \emph{constellation} for short, is a map with two types of faces \emph{black} and \emph{white}, and $k$ types of vertices $1,2,\ldots,k$, such that: 
\begin{compactitem}
\item[(i)] each edge separates a black face and a white face, 
\item[(ii)] each black face has degree $k$ and is incident to vertices of type $1,2,\ldots,k$ in this order clockwise around the face. 
\end{compactitem}
Two constellations are shown in Figure~\ref{fig:exsconstellations}. The black faces are also called \emph{hyperedges}. The \emph{size} of a constellation is the number of hyperedges. A constellation of size $n$ is \emph{labelled} if its hyperedges receive distinct labels in~$[n]$.

We now recall the link between constellations and products of permutations. We call $k$-hypergraph a pair $G=(V,E)$ where $V$ is a set of vertices, each of them having a \emph{type} in $[k]$, and $E$ is a set of \emph{hyperedges} which are subsets of $V$ containing exactly one vertex of each type. A \emph{rotation-system} for the hypergraph $G$ is an assignment for each vertex $v$ of a cyclic order of the hyperedges \emph{incident to} $v$ (i.e., containing $v$). Clearly each $k$-constellation defines a connected $k$-hypergraph together with a rotation system (the clockwise order of the hyperedges around each vertex). In fact Lemma~\ref{lem:embedding-map} readily implies the following result. 
\begin{lemma}\label{lem:k-embedding}
For any connected $k$-hypergraph $G$, there is a bijection between $k$-constellations of underlying $k$-hypergraph $G$ and the rotation systems of~$G$. 
\end{lemma}
Now given a hyperedge-labelled $k$-constellation $C$ of size $n$, we define some permutations $\pi_1,\ldots,p_k$ as follows: for each $t\in[k]$ we define the cycles of the permutation $\pi_t$ to be the counterclockwise order of the hyperedges around the vertices of type $t$. Examples are given in Figure~\ref{fig:exsconstellations}. We then say that the hyperedge-labelled $k$-constellation $C$ \emph{represents} the tuple $\varrho(C)=(\pi_1,\ldots,\pi_k)$. From Lemma~\ref{lem:k-embedding} it is easy to establish the following classical result (see e.g. \cite{LZ}).
\begin{lemma}\label{lem:embedding}
The representation mapping $\varrho$ is a bijection between hyperedge-labelled $k$-constellations of size $n$ and tuples of permutations $(\pi_1,\ldots,\pi_k)$ of $[n]$ acting transitively on $[n]$. Moreover the number of white faces of the constellation is equal to the number of cycles of the product $\pi_1\pi_2\cdots\pi_k$. 
\end{lemma}

An edge of a constellation has \emph{type} $t\in[k]$ if its endpoints have types $t$ and $t+1$ (the types of the vertices and edges are considered modulo $k$). A $k$-constellation has \emph{type} $(p_1,\ldots,p_k)$ if it has $p_t$ vertices of type $t$ for all $t\in[k]$. The \emph{hyperdegree} of a vertex is the number of incident hyperedges.
A constellation of type $(p_1,\ldots,p_k)$ is \emph{vertex-labelled} if for each $t\in[k]$ the $p_t$ vertices of type $t$ have distinct labels in $[p_t]$. 
We say that such a constellation has \emph{vertex-compositions} $(\ga^{(1)},\ldots,\ga^{(k)})$ if for all $t\in[k]$, $\ga^{(t)}$ is a composition of size $n$ and length $p_t$ whose $i$th part is the hyperdegree of the vertex of type $t$ labelled $i$.

A $k$-constellation is \emph{rooted} if one of its hyperedges is distinguished as the \emph{root hyperedge}. The vertex of type $k$ incident to the root hyperedge is called \emph{root vertex}. There are $n!$ distinct ways of labelling a rooted constellation of size $n$ (because a rooted constellation has no symmetry preserving the root hyperedge). Hence, there is a 1-to-$(n-1)!$ correspondence between rooted constellations of size $n$ and hyperedge-labelled constellations of size $n$.

A $k$-\emph{cactus} is a $k$-constellation with a single white face. By Lemma~\ref{lem:embedding} the hyperedge-labelled $k$-cacti correspond bijectively to the factorizations of one of the $(n-1)!$ long cycles into $k$ factors (transitivity is redundant in this case), while rooted cacti correspond bijectively to the factorizations of the permutation $(1,2,\ldots,n)$. Since Jackson's counting formula is about \emph{colored factorizations} of $(1,2,\ldots,n)$ (see Definition~\ref{def:colored-factorization}), we now consider \emph{vertex-colored cacti}. Given some positive integers $q_1,\ldots,q_k$, a $(q_1,\ldots,q_k)$\emph{-colored cacti} is a $k$-cacti together with an assignment of \emph{colors} to vertices, such that for every $t\in[k]$ the vertices of type $t$ are colored using every color in $[q_t]$. A $(2,1,3)$-colored cacti is represented in Figure~\ref{fig:exscolconstellations}. The \emph{color-compositions} of a $(q_1,\ldots,q_k)$-colored cacti of size $n$ is the tuple $(\ga^{(1)},\ldots,\ga^{(k)})$, where for all $t\in[k]$, $\ga^{(t)}$ is a composition of size $n$ and length $q_t$ whose $i$th part is the number of hyperedges incident to vertices of type $t$ colored $i$. It is clear from the representation mapping $\varrho$, that $(q_1,\ldots,q_k)$-colored cacti of color-compositions $(\ga^{(1)},\ldots,\ga^{(k)})$ are in bijection with the $(q_1,\ldots,q_k)$-colored factorizations of $(1,2,\ldots,n)$ with color-compositions $(\ga^{(1)},\ldots,\ga^{(k)})$.\\

\begin{figure}[h] 
\begin{center} 
\includegraphics[width=.7\linewidth]{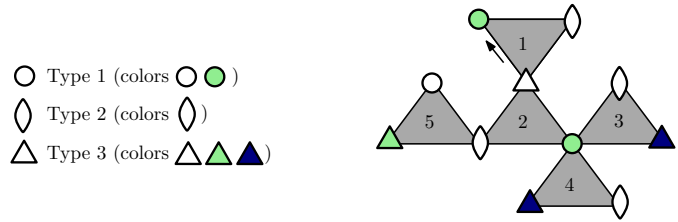} 
\caption{A $(2,1,3)$-colored cacti (embedded in the sphere) with color-compositions $(\ga^{(1)},\ga^{(2)},\ga^{(3)})$, where $\ga^{(1)}=(1,4)$, $\ga^{(2)}=(5)$ and $\ga^{(3)}=(2,1,2)$.}\label{fig:exscolconstellations}
\end{center} 
\end{figure} 

From now on, all our results and proofs are stated in terms of constellations and cacti.


\section{From cacti to tree-rooted constellations}\label{sec:tree-rooted}
In this section we establish a bijection between vertex-colored cacti and certain constellations with a distinguished spanning tree.
Let $C$ be a $k$-constellation and let $v_0$ be a vertex. We call \emph{$v_0$-arborescence} of $C$ a spanning tree $A$ such that every vertex $v\neq v_0$ of type $t$ is incident to exactly one edge of type $t$ in $A$ (equivalently, the spanning tree $A$ is oriented from the leaves toward $v_0$ by orienting every edge of $A$ of type $t\in[k]$ from its endpoint of type $t$ toward its endpoint of type $t+1$). A \emph{tree-rooted constellation} is a pair $(C,A)$ made of a rooted constellation $C$ together with a $v_0$-arborescence $A$, where $v_0$ is the root vertex of $C$. An example of tree-rooted constellation is given in Figure~\ref{fig:exBESTthm} (bottom right).

\begin{theorem}\label{thm:bij-tree-rooted}
Let $p_1,\ldots,p_k$ be positive integers. There is a bijection $\Phi$ between the set $\mC^n_{p_1,\ldots,p_k}$ of $(p_1,\ldots,p_k)$-colored rooted $k$-cacti of size $n$ (these encode the $(p_1,\ldots,p_k)$-colored factorizations of $(1,2,\ldots,n)$), and the set $\mT^n_{p_1,\ldots,p_k}$ of vertex-labelled tree-rooted $k$-constellations of size $n$ and type $(p_1,\ldots,p_k)$. 

Moreover, the bijection has the following \emph{degree preserving property}: for any vertex-colored cactus $C$, the number of edges joining vertices of type $t$ and color $i$ to vertices of type $t+1$ and color $j$ in $C$ is equal to the number of edges joining the vertex of type $t$ labelled $i$ to the vertex of type $t+1$ labelled $j$ in the tree-rooted constellation $\Phi(C)$.
\end{theorem}

\noindent \textbf{Remark.}
The degree preserving property of Theorem~\ref{thm:bij-tree-rooted} implies that for any tuple of compositions $(\ga^{(1)},\ldots,\ga^{(k)})$, the mapping $\Phi$ establishes a bijection between cacti of color-compositions $(\ga^{(1)},\ldots,\ga^{(k)})$ and tree-rooted constellations of vertex-compositions $(\ga^{(1)},\ldots,\ga^{(k)})$.\\

\noindent \textbf{Remark.} In the case $k=2$ the tree-rooted $k$-constellations can be identified with rooted bipartite maps with a distinguished spanning tree (simply by considering the hyperedges as edges). These objects are easy to count (see \cite{OB:Harer-Zagier-non-orientable}), so that the case $k=2$ of Theorem~\ref{thm:Jackson-counting} follows easily from Theorem~\ref{thm:bij-tree-rooted} in this case.\\ 

The remaining of this section is devoted to the proof of Theorem~\ref{thm:bij-tree-rooted}. Our strategy parallels the one developed in \cite{OB:Harer-Zagier-non-orientable} (building on some ideas of Lass \cite{Lass:Harer-Zagier}) in order to prove extensions of the Harer-Zagier formula. This proof is illustrated in Figure~\ref{fig:exBESTthm}. We shall recombine the information given by a vertex-colored cactus into the information given by a tree-rooted constellation through the \emph{BEST Theorem} (see Lemma~\ref{lem:BESTthm} below). 

\begin{figure}[h] 
\begin{center} 
\includegraphics[width=\linewidth]{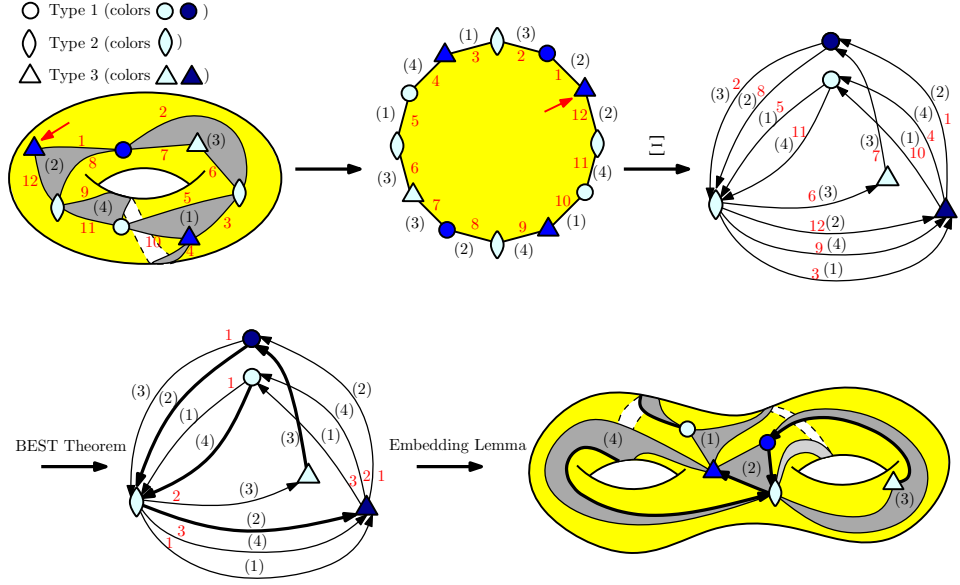} 
\caption{From a vertex-colored cactus to a tree-rooted constellation via the BEST Theorem.} \label{fig:exBESTthm}
\end{center} 
\end{figure}

We call \emph{$k$-digraph} a directed graph with $k$ types of vertices $1,\ldots,k$, such that every vertex has as many ingoing and outgoing arcs, and every arc goes from a vertex of type $t$ to a vertex of type $t+1$ for some $t\in[k]$ (as usual the types of vertices are considered modulo $k$). An arc going from a vertex of type $t$ to a vertex of type $t+1$ is said to have \emph{type} $t$. Note that a $k$-digraph has as many arcs of each type, and we say that it has \emph{size} $n$ if it has $n$ arcs of each type. An \emph{arc-labelling} of a $k$-digraph of size $n$ is an assignment of distinct labels in $[n]$ to the $n$ arcs of type $t$, in such a way that for any $(t,i)\in[k]\times [n]$ the end of the arc of type $t$ and label $i$ is the origin of the arc of type $t+1$ and label $i$. Observe that arc-labelled $k$-digraphs easily identify with hyperedge-labelled $k$-hypergraphs. A $k$-digraph has \emph{type} $(p_1,\ldots,p_k)$ if for each $t\in[k]$ there are $p_t$ vertices of type $t$. It is \emph{vertex-labelled} by assigning distinct labels in $[p_t]$ to its $p_t$ vertices of type $t$ for all $t\in[k]$.

\begin{lemma} \label{lem:cactus-to-Eulerian}
There is a bijection $\Xi$ between the set of hyperedge-labelled rooted $(p_1,\ldots,p_k)$-colored cacti of size $n$, and the set of pairs $(G,\eta)$ where $G$ is a arc-labelled vertex-labelled $k$-digraph of type $(p_1,\ldots,p_k)$ and $\eta$ is an Eulerian tour of $G$ starting and ending at a vertex of type $k$. 
\end{lemma}

Lemma~\ref{lem:cactus-to-Eulerian} is illustrated in the top part of Figure~\ref{fig:exBESTthm}.

\begin{proof}
We call \emph{black $k$-gon} a polygon with $k$ vertices of type $1,2,\ldots,k$ in clockwise order, and \emph{white $kn$-gon} a polygon with $kn$ vertices, such that the type of vertices increases by one (modulo $k$) along each edge in counterclockwise order (modulo $k$). A white $kn$-gon is \emph{rooted} if a corner incident to a vertex of type $k$ is distinguished as the \emph{root-corner}; it is $(p_1,\ldots,p_k)$-colored if for all $t\in[k]$ the vertices of type $t$ are colored using every color in $[p_t]$.

Observe that the $n$ hyperedges of a $k$-cactus of size $n$ are black $k$-gons, while its white face is a white $kn$-gon (since faces of cactus are simply connected). Moreover the $k$-cactus is completely determined (up to homeomorphism) by specifying the gluing of the black $k$-gons with the white $kn$-gon (that is specifying the pair of edges to be identified). Thus, a rooted hyperedge-labelled $(p_1,\ldots,p_k)$-colored cactus is obtained by taking a rooted $(p_1,\ldots,p_k)$-colored white $kn$-gon, and gluing its edges to the edges of $n$ labelled black $k$-gon \emph{so as to respect the color and type of the vertices} (certain vertices of the white $kn$-gon are identified by the gluing). Now, a rooted $(p_1,\ldots,p_k)$-colored white $kn$-gon is bijectively encoded by a pair $(\tilde{G},\eta)$, where $\tilde{G}$ is a vertex-labelled $k$-digraph of type $(p_1,\ldots,p_k)$ and $\eta$ is an Eulerian tour of $\tilde{G}$ (the Eulerian tour gives the order of the colors around the white $kn$-gon in counterclockwise direction starting from the root-corner). Moreover, the gluings of the $n$ labelled black $k$-gons (respecting the type and coloring) are in bijection with the arc-labellings of $\tilde{G}$. This establishes the claimed bijection.
\end{proof}

We now recall the BEST Theorem for Eulerian tours\footnote{This Theorem is due to de Bruijn, van Aardenne-Ehrenfest, Smith and Tutte. See \cite[Theorem 5.6.2]{EC2} for a proof.}.
Let $G$ be a directed graph and let $v_0$ be a vertex. We call \emph{$v_0$-Eulerian tour} an Eulerian-tour starting and ending at vertex $v_0$. Observe that a $v_0$-Eulerian tour is completely characterized by its \emph{local-order}, that is, the assignment for each vertex $v$ of the order in which the outgoing edges incident to $v$ are used. Note however that not every local order corresponds to an Eulerian tour. We call \emph{$v_0$-arborescence} a spanning tree $A$ of $G$ oriented from the leaves toward $v_0$ (i.e., every vertex $v\neq v_0$ has exactly one outgoing arc in $A$). 

\begin{lemma}[BEST Theorem]\label{lem:BESTthm}
Let $G$ be an arc-labelled directed graph where every vertex has as many ingoing arcs as outgoing ones, and let $v_0$ be a vertex of $G$. 
A local order corresponds to a $v_0$-Eulerian tour if and only if the set of last outgoing arcs out of the vertices $v\neq v_0$ form a $v_0$-arborescence. Consequently, there is a bijection between the set of $v_0$-Eulerian tours of $G$ and the set of pairs $(A,\tau)$, where $A$ is a $v_0$-arborescence, and $\tau$ is an assignment for each vertex $v$ of a total order of the incident outgoing arcs not in $A$. 
\end{lemma}

We now complete the proof of Theorem~\ref{thm:bij-tree-rooted}.
By combining Lemma~\ref{lem:cactus-to-Eulerian} and the BEST Theorem, one gets a bijection between rooted hyperedge-labelled $(p_1,\ldots,p_k)$-colored cacti and triples $(G,A,\theta)$ where $G$ is an arc-labelled vertex-labelled $k$-digraph of type $(p_1,\ldots,p_k)$, $A$ is a $v_0$-arborescence of $G$ for a vertex $v_0$ of type $k$, and $\tau$ is an assignment for each vertex $v$ of a total order of the arcs not in $A$ going out of $v$. Observe that $\tau$ encodes the same information as a pair $(a_0,\tau')$, where $a_0$ is an arc going out of $v_0$ and $\tau'$ is an assignment for each vertex $v$ of a cyclic order of the arcs going out of $v$. Now the arc-labelled vertex-labelled $k$-digraph $G$ encodes the same information as a hyperedge-labelled vertex-labelled $k$-hypergraph $G'$, and $\tau'$ can be seen as a rotation system for $G'$. Thus, by Lemma~\ref{lem:k-embedding} the pair $(G,\tau)$ encodes the same information as a rooted hyperedge-labelled vertex-labelled $k$-constellation $C$ of type $(p_1,\ldots,p_k)$ (note that the hypergraph $G'$ is clearly connected since it has an arborescence $A$). Lastly, the $v_0$-arborescence $A$ of $G$ clearly encodes a $v_0$-arborescence of the constellation $C$, where $v_0$ is the root vertex of $C$. 

We thus have obtained a bijection between rooted hyperedge-labelled $(p_1,\ldots,p_k)$-colored cacti and the hyperedge-labelled vertex-labelled tree-rooted constellations. The labelling of the hyperedges can actually be disregarded since there are $n!$ distinct ways of labelling the hyperedges of a rooted constellation of size $n$. This gives the bijection announced in Theorem~\ref{thm:bij-tree-rooted}. Moreover it is easy to check that it has the claimed \emph{degree preserving property}. \hfill $\square$


\section{Symmetries for tree-rooted constellations} \label{sec:symmetry}
In this section we prove that for vertex-labelled tree-rooted constellations of a given type $(p_1,\ldots,p_k)$, every vertex-compositions is equally likely. This together with Theorem~\ref{thm:bij-tree-rooted} proves the symmetry property stated in Theorem~\ref{thm:symmetry}.

We denote by $\mT_{\ga^{(1)},\ldots,\ga^{(k)}}$ the set of vertex-labelled tree-rooted constellations of vertex-compositions $(\ga^{(1)},\ldots,\ga^{(k)})$. 

\begin{thm}\label{thm:sym-tree-rooted}
If $\ga^{(1)},\ldots,\ga^{(k)},{\de}^{(1)},\ldots,\de^{(k)}$ are compositions of $n$ such that $\ell(\ga^{(t)})=\ell(\de^{(t)})$ for all $t\in[k]$, then the sets $\mT_{\ga^{(1)},\ldots,\ga^{(k)}}$ and $\mT_{\de^{(1)},\ldots,\de^{(k)}}$ are in bijection. 
\end{thm}

\noindent \textbf{Remark.} Theorem~\ref{thm:sym-tree-rooted} gives the hope of counting tree-rooted constellations of given type, by looking at the simplest possible vertex-compositions. For instance, one can try to enumerate the set $\mT_{\ga^{(1)},\ldots,\ga^{(k)}}$ where $\ga^{(t)}=(n-p_t+1,1,1,\ldots,1)$ for all $t\in[k]$ (similar ideas lead to a very easy way of counting $k$-cacti embedded in the sphere \cite{OB-AM:planar-cacti}). However, our efforts in this direction only led to a restatement of Jackson counting formula as a probabilistic puzzle similar to Theorem~\ref{thm:tree-puzzle} which we could not easily solve for $k\geq 3$.


\begin{proof}
Let $t\in[k]$ and $i,j\in[p_t]$. In order to prove Theorem~\ref{thm:sym-tree-rooted} it suffices to exhibit a bijection $\varphi_{t,i,j}$ between $\mT_{\ga^{(1)},\ldots,\ga^{(k)}}$ and $\mT_{\de^{(1)},\ldots,\de^{(k)}}$ when $\ga^{(s)}={\de}^{(s)}$ for all $s\neq t$, $\ga^{(t)}_x=\de^{(t)}_x$ for all $x\neq i,j$, $\ga^{(t)}_i-1=\de^{(t)}_i$ and $\ga^{(t)}_j+1=\de^{(t)}_j$. In other words, we want to construct a bijection $\varphi_{t,i,j}$ which decreases by one the hyperdegree of the vertex of type $t$ labelled $i$ and increases by one the hyperdegree of the vertex of type $t$ labelled $j$. Recall from Lemma~\ref{lem:k-embedding} that a $k$-constellation is defined by a (connected) $k$-hypergraph together with a rotation system (clockwise order of hyperedges around the vertices); therefore it is well defined to \emph{unglue} a hyperedge from a vertex of type $t$ and \emph{reglue} it in a specified corner of another vertex of type $t$. We will use these operations to define the mapping $\varphi_{t,i,j}$ below; see Figure~\ref{fig:varphi}. 

\begin{figure}[h] 
\begin{center} 
\includegraphics[width=\linewidth]{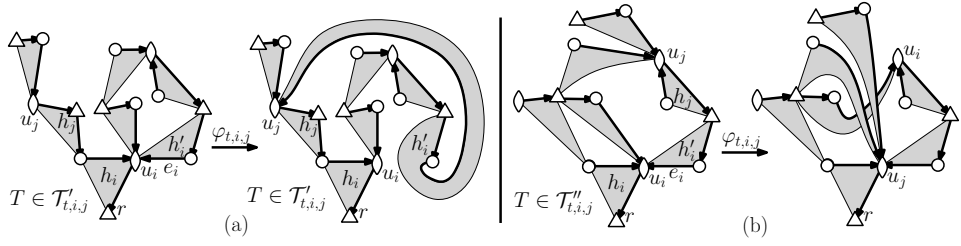} 
\caption{The bijection $\varphi_{t,i,j}$ applied to a tree-rooted constellation in $\mT'_{t,i,j}$ (left), or in $\mT''_{t,i,j}$ (right). The tree-rooted $k$-constellations are represented as $k$-hypergraphs together with a rotation system (so the overlappings of the hyperedges in this figure are irrelevant).}\label{fig:varphi}
\end{center} 
\end{figure}

Let $\mT_{t,i}$ be the set of vertex-labelled tree-rooted constellations of type $(p_1,\ldots,p_k)$ such that the vertex of type $t$ labelled $i$ has hyperdegree at least two. Let $T$ be a tree-rooted constellation in $\mT_{t,i}$, let $u_i$ and $u_j$ be the vertices of type $t$ labelled $i$ and $j$ respectively, let $r$ be the root vertex, and let $A$ be the marked $r$-arborescence. 
If $u_i\neq r$ we denote by $h_i$ be hyperedge incident to the edge joining $u_i$ to its parent in $A$, while if $u_i=r$ we denote by $h_i$ the the root hyperedge. We define $h_j$ similarly. Let $h_i'$ be the hyperedge preceding $h_i$ in clockwise order around $u_i$ and let $e_i$ be the edge of type $t-1$ incident to $h_i'$. Observe that $h_i\neq h_i'$ since the hyperdegree of $u_i$ is at least two. 

In order to define the mapping $\varphi_{t,i,j}$ we need to consider two cases which are illustrated in Figure~\ref{fig:varphi}. We first define a partition $\mT_{t,i}=\mT_{t,i,j}'\cup \mT_{t,i,j}''$ by declaring that $T$ is in $\mT_{t,i,j}'$ if the edge $e_i$ is not on the path from $u_j$ to the root vertex $r$ in the arborescence $A$, and that $T$ is in $\mT_{t,i,j}''$ otherwise. Suppose first that $T$ is in $\mT_{t,i,j}'$. In this case we define $\varphi_{t,i,j}(T)$ as the constellation (with marked edges) obtained from the tree-rooted constellation $T$ (with marked edges corresponding to the arborescence $A$) by ungluing the hyperedge $h_i'$ from $u_i$ and gluing it to $u_j$ in the corner preceding the hyperedge $h_j$ in clockwise order around $u_j$; see Figure~\ref{fig:varphi}(a). Observe that $\varphi_{t,i,j}(T)$ is a tree-rooted constellation (in particular the marked edges form an $r$-arborescence $A'$ of $\varphi_{t,i,j}(T)$). Moreover $\varphi_{t,i,j}(T)$ is in $\mT_{t,j}$ and more precisely in $\mT_{t,j,i}'$. It is also easy to see that $\varphi_{t,j,i}(\varphi_{t,i,j}(T))=T$. Suppose now that $T$ is in $\mT_{t,i,j}''$. In this case we define $\varphi_{t,i,j}(T)$ as the constellation (with marked edges) obtained from $T$ (with marked edges corresponding to the arborescence $A$) as follows: we unglue all the hyperedges incident to $u_i$ except $h_i$ and $h_i'$, we unglue all the hyperedges incident to $u_j$ except $h_j$, we reglue the hyperedges unglued from $u_j$ to $u_i$ in the corner preceding $h_i'$ in clockwise order around $u_i$ (without changing their clockwise order), we reglue the hyperedges unglued from $u_i$ to $u_j$ (in the unique possible corner), and lastly we exchange the labels $i$ and $j$ of the vertices $u_i$ and $u_j$; see Figure~\ref{fig:varphi}(b). It is easy to see that $\varphi_{t,i,j}(T)$ is a tree-rooted constellation (in particular the marked edges form an $r$-arborescence of $\varphi_{t,i,j}(T)$). Moreover $\varphi_{t,i,j}(T)$ is in $\mT_{t,j}$ and more precisely in $\mT_{t,j,i}''$. It is also easy to see that $\varphi_{t,j,i}(\varphi_{t,i,j}(T))=T$. 

We have shown that $\varphi_{t,i,j}$ is a mapping from $\mT_{t,i}$ to $\mT_{t,j}$. Moreover $\varphi_{t,j,i}\circ\varphi_{t,i,j}=Id$ for all $i,j$, thus $\varphi_{t,i,j}=\varphi_{t,j,i}^{-1}$ is a bijection. Lastly, the bijection $\varphi_{t,i,j}$ decreases by one the hyperdegree of the vertex of type $t$ labelled $i$ and increases by one the degree of the vertex of type $t$ labelled $j$. Thus $\varphi_{t,i,j}$ has all the claimed properties.
\end{proof}


\section{From tree-rooted constellations to nebulas}\label{sec:nebulas}
In Section~\ref{sec:tree-rooted} we obtained a bijection between vertex-colored cacti and tree-rooted constellations. 
In this section we take another look at tree-rooted constellations by characterizing their \emph{duals}. We eventually obtain a bijection between a class of tree-rooted constellations and certain decorated maps with a single face called \emph{nebulas}.

We call \emph{tree-pointed $k$-constellation} a pair $(C,A)$, where $C$ is a rooted $k$-constellation and $A$ is a $v_0$-arborescence for some vertex $v_0$ which can be distinct from the the root-vertex of $C$. If the constellation $C$ has type $(p_1,\ldots,p_k)$ and $v_0$ has type $t$, then the tree-pointed $k$-constellation $(C,A)$ is said to have \emph{reduced type} $(p_1',\ldots,p_k')$, where $p_t'=p_t-1$ and $p_s'=p_s$ for $s\neq t$. Observe that the arborescence $A$ has $p_t'$ edges of type $t$ for all $t\in [k]$ (because for every vertex $v\neq v_0$ of type $t$ the edge of $A$ joining $v$ to its parent has type $t$).
\begin{lemma}\label{lem:pointing}
There is a $1$-to-$\prod_{i=1}^k p_i!$ correspondence between tree-pointed $k$-constellations of reduced type $(p_1,\ldots,p_k)$, and the union $\mT^n_{p_1+1,p_2,\ldots,p_k } \cup \mT^n_{p_1,p_2+1,\ldots,p_k } \cup \cdots \cup \mT^n_{p_1,p_2,\ldots,p_k+1}$ of vertex-labelled tree-rooted constellations. 
\end{lemma}

\begin{proof}
We first claim that for any constellation $C$ and any vertices $u,v$, there are as many $u$-arborescences as $v$-arborescences. 
Indeed, if one orients the edges of type $t$ of $C$ toward their endpoint of type $t+1$, one gets a \emph{Eulerian digraph} (oriented graph with as many ingoing and outgoing edges at each vertex). Moreover it is an easy corollary of the BEST Theorem that Eulerian digraphs have the same number arborescences directed toward each vertex (see \cite[Cor 5.6.3]{EC2}). Thus, the tree-pointed constellations of reduced type $(p_1,\ldots,p_k)$ such that the root vertex of the arborescence has type $t$ are equinumerous to the tree-rooted constellations of type $(p_1,\ldots,p_t+1,\ldots,p_k)$ with a distinguished vertex of type $t$. Since there are $(p_t+1)$ ways of distinguishing a vertex of type $t$ in such a tree-rooted constellation versus $(p_t+1)\prod_{i=1}^k p_i!$ ways of labelling its vertices, one gets the claimed correspondence.
\end{proof}

We now consider the dual of constellations. Recall that the \emph{dual} of a map $M$ is the map $M^*$ obtained by placing a vertex of $M^*$ in each face of $M$ and drawing an edge of $M^*$ across each edge of $M$. Duality is a genus preserving involution on maps such that the vertices, edges and faces of $M$ correspond respectively to the faces, edges and vertices of $M^*$. We now describe the dual of $k$-constellations. Observe that $k$-constellations can be characterized as the maps with black and white faces, and $k$ types of edges $1,2,\ldots,k$ such that 
\begin{compactitem}
\item[(i)] each edge separates a black and a white face, 
\item[(ii)] each black face has degree $k$, 
\item[(iii)] in clockwise order around a hyperedge (resp. white face) the type of the edges increases (resp. decreases) by one from one edge to the next.
\end{compactitem}
Indeed, with the preceding conditions, for each vertex $v$ there exists $t\in[k]$ such that the edges incident to $v$ are alternatingly of type $t-1$ and $t$ for some $t$ in $[k]$ (and we can thus say that $v$ has type $t$ in this case). With the previous characterization, it is clear that duality gives a bijection between $k$-constellations and the \emph{dual-constellations}, which are defined as the maps with black and white vertices, and $k$ types of edges $1,2,\ldots,k$ such that 
\begin{compactitem}
\item[(i)] each edge joins a black and a white vertex 
\item[(ii)] each black vertex has degree $k$,
\item[(iii)] in clockwise order around a black vertex (resp. white vertex) the type of the edges increases (resp. decreases) by one from one edge to the next. 
\end{compactitem}

We will now describe a class of maps closely related to the dual of tree-rooted constellations. A \emph{bud} is a dangling half-edge, that is, a half-edge which is not part of a complete edge. 
We define a $k$\emph{-nebula}, or \emph{nebula} for short, as a map having a \emph{single face} with black and white vertices, $k$ types of edges $1,\ldots,k$, and $k$ types of buds $1,\ldots,k$ such that 
\begin{compactitem}
\item[(i)] each edge joins a black and a white vertex,
\item[(ii)] each black vertex has degree $k$, 
\item[(iii)] in clockwise order around a black vertex (resp. white vertex) the type of the edges or buds increases (resp. decreases) by one from one half-edge to the next,
\item[(iv)] for all $t\in[k]$ the number of buds of type $t$ incident to black vertices is equal to the number of buds of type $t$ incident to white vertices. 
\end{compactitem}
A 3-nebula is shown in Figure~\ref{fig:dual-opening} (right).
A nebula is \emph{rooted} if one of the black vertices is distinguished as the \emph{root vertex}. We call \emph{black buds} and \emph{white buds} respectively the buds incident to black and to white vertices. A $k$-nebula is said to have \emph{size} $n$ and \emph{type} $(p_1,\ldots,p_k)$ if it has $n$ black vertices and $p_t$ black buds of type $t$ for all $t\in[k]$. 

Consider a $k$-constellation $C$ and a spanning tree $A$. We call \emph{dual-opening} of $(C,A)$ the map with buds $N$ obtained from the dual-constellation $C^*$ by cutting in two halves the edges of $C^*$ which are crossing the edges of the tree $A$ (the edges of $C^*$ of type $t$ crossing $A$ gives two buds of type $t$). We call \emph{root vertex} of $N$ the dual of the root hyperedge of $C$ . The dual-opening of a tree-pointed constellation $(C,A)$ is represented in Figure~\ref{fig:dual-opening}.\\

\begin{figure}[h] 
\begin{center} 
\includegraphics[width=\linewidth]{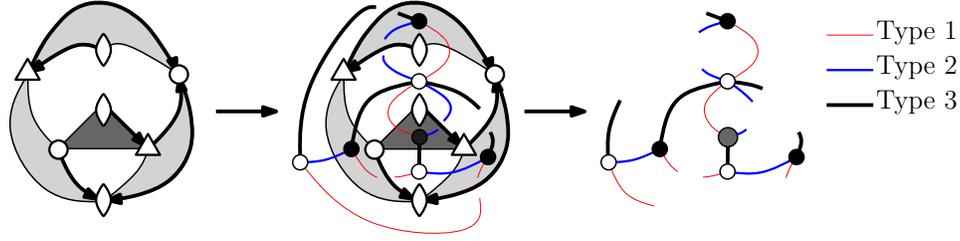} 
\caption{Dual-opening of a tree-pointed constellation. Here the tree-constellation is embedded in the sphere (hence so is its dual opening). The types of edges are indicated by their thicknesses. The root-hyperedge is indicated in darker shade, and the corresponding root-vertex of the nebula is represented in gray.} \label{fig:dual-opening}
\end{center} 
\end{figure} 

We now state the main result of this section.

\begin{theorem}\label{thm:bij-nebulas}
The dual-opening gives a bijection between tree-pointed constellations of reduced type $(p_1,\ldots,p_k)$ and rooted nebulas of type $(p_1,\ldots,p_k)$.
\end{theorem}

Theorems~\ref{thm:bij-tree-rooted},~\ref{thm:bij-nebulas} and Lemma~\ref{lem:pointing} immediately imply the following result.

\begin{corollary} \label{cor:bij-nebulas}
The nebulas of size $n$ and type $(p_1,\ldots,p_k)$ are in $1$-to-$\prod_{i=1}^k p_i!$ correspondence with the disjoint union
$\mC^n_{p_1+1,p_2,\ldots,p_k } \cup \mC^n_{p_1,p_2+1,\ldots,p_k } \cup \cdots \cup \mC^n_{p_1,p_2,\ldots,p_k+1}$ of colored cacti. 
\end{corollary}

\noindent \textbf{Remark.} By definition, tree-rooted constellations are tree-pointed constellations such that the marked arborescence is directed toward the root-vertex of the map. One can check that tree-rooted constellations actually correspond to nebulas such that the sequence of black buds and white buds around $N$ form a \emph{parenthesis system}, that is, the number of black buds never exceeds the number of white buds when turning in clockwise direction around the face of $N$ starting at the corner preceding the edge of type 1 around the root vertex. Our reason for considering tree-pointed constellations is to get get rid of the parenthesis system condition (which is hard to control). \\

The rest of this section is devoted to the proof of Theorem~\ref{thm:bij-nebulas}. 

\begin{lemma}\label{lem:opening}
Let $C$ be a $k$-constellation and let $A$ be a spanning tree (not necessarily an arborescence) having $p_t$ edges of type $t\in[k]$. Then the dual-opening of $(C,A)$ is a nebula of type $(p_1,\ldots,p_k)$. 
\end{lemma}

\begin{proof}
Let $N$ be the dual opening of $(C,A)$. Since the spanning tree $A$ connects all the vertices of the constellation $C$, the faces of $C^*$ are all merged into a single face of $N$ by cutting the edges of $C^*$ crossed by $A$. Moreover, this face of $N$ is simply connected because $A$ is simply connected (i.e., has no cycle). Hence $N$ is indeed a \emph{map} with a \emph{single face}. The other properties of nebulas are easily seen to hold.
\end{proof}

Lemma~\ref{lem:opening} proves that the dual opening of a tree-pointed constellation is a nebula. 
We now define the closure of nebulas.
Let $N$ be a nebula. We call \emph{turning clockwise around $N$} the process of walking around the face of $N$ by following its edges, with the edges on the right side of the walker (see Figure~\ref{fig:bidding}(a)).
We say that a white bud $w$ \emph{matches} a black bud $b$ if there is no bud between $w$ and $b$ when turning clockwise around $N$ starting from $w$. We call \emph{closure} of the nebula $N$, the result of \emph{recursively} forming edges by gluing together pairs of matching buds (thus, at a later step of this recursive process, we say that a white bud $w$ matches a black bud $b$, if there is no bud between $w$ and $b$ when turning clockwise around the face containing the buds). This process is illustrated in Figure~\ref{fig:closure}. 
It is clear that the closure can be done without creating any edge-crossings, that it will exhaust all buds, and that the result is uniquely defined. 
\begin{lemma}\label{lem:closure}
The pairs of buds glued together during the closure of a nebula have the same type. Consequently, the closure of a nebula gives a dual-constellation.
\end{lemma}

\begin{proof}
It is easy to see that if a map satisfies Condition (iii) of nebulas, then matching pairs of buds $(w,b)$ have the same type. Moreover, in this case, Condition (iii) is preserved by forming an edge out of the buds $w,b$. This shows the first claim by induction on the number of buds. The second claim is clear. 
\end{proof}

\begin{figure}[h] 
\begin{center} 
\includegraphics[width=\linewidth]{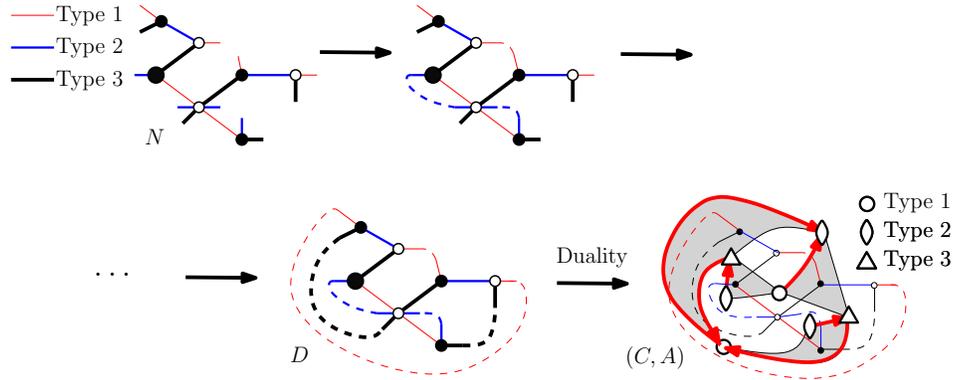} 
\caption{The closure of a nebula $N$ (here $N$ is embedded in the sphere) obtained by recursively gluing pairs of matching edges gives a dual-constellation $D$ (with dashed bud-edges). Taking its dual gives a tree-rooted constellation $(C,A)$.} \label{fig:closure}
\end{center} 
\end{figure} 

Let $N$ be a nebula, let $D$ be its closure (which is a dual-constellation) and let $C$ be the corresponding constellation. Let $A$ be the set of edges of the constellation $C$ which are dual to the edges of $D$ which have been created during the closure of $N$ (by joining two buds). The pair $(C,A)$ is called the \emph{dual-closure} of the nebula $N$.

\begin{lemma}\label{lem:dual-closure}
The dual-closure of a nebula is a tree-pointed constellation. 
\end{lemma}

\begin{proof}
Let $N$ be a nebula, let $D$ be its closure, and let $(C,A)$ be its dual-closure. By Lemma~\ref{lem:closure} we know that $C$ is a constellation, so we only need to show that $A$ is an arborescence of $C$. 
We call bud-edges the edges of $D$ created during the closure and we view them as oriented from the white bud to the black bud (so that we can distinguish the face on their \emph{left} and the face on their \emph{right}). Let $f_0$ be the face of the nebula $N$. During the closure of $N$, each time a matching pair of buds are glued into a bud-edge $e$, we consider the face at the right of $e$ as a ``new face'', while we consider the face the left of $e$ as the ``original face'' $f_0$. Hence, at any time during the closure, the original face $f_0$ is at the left of every incident bud-edge, while any face $f\neq f_0$ is at the right of exactly one bud-edge which we call the \emph{closing edge of $f$}. Observe that for any face $f\neq f_0$ of $D$, one can reach the face $f_0$ of $D$ by starting inside $f$ and repeatedly crossing the closing edge of the current face.

Let $v_0$ be the vertex of the constellation $C$ dual to the face $f_0$ of $D$, and for any vertex $v\neq v_0$ let the \emph{parent edge} of $v$ be the dual of the closing edge of $f$, where $f$ is the face of $D$ dual to the vertex $v$. By definition, $A$ is the set of parent edges and we will show that it is a $v_0$ arborescence. First of all, the above observation implies that starting from a vertex $v\neq v_0$ and repeatedly following the parent edges one eventually reaches the vertex $v_0$. Thus $A$ connects all the vertices of $C$ and since there is one more vertex in $C$ than edges in $A$, we can conclude that $A$ is a spanning tree of $C$. It now suffices to check that the tree $A$ is oriented toward $v_0$, or equivalently that the parent edge of a vertex of type $t$ has type $t$. This is true because of the orientation convention for closing edges and the fact that the type increases in clockwise order around the black vertices of the dual-constellations $D$. 
\end{proof}

We now conclude the proof of Theorem~\ref{thm:bij-nebulas}. Let $\Lambda$ and $\Delta$ denote respectively the dual-opening and dual-closure mappings. By Lemma~\ref{lem:dual-closure}, the closure $\Delta(N)$ of any nebula $N$ is a tree-pointed constellation. Moreover, it is clear that $\Lambda\circ\Delta(N)=N$. 

We now consider a tree-pointed constellation $(C,A)$, where $A$ denotes the marked $v_0$-arborescence. By Lemma~\ref{lem:opening}, $N:=\Lambda(C,A)$ is a nebula. We want to prove that the  dual-closure of $N$ is the tree-pointed constellation $(C,A)$. 
For this, we consider the white bud $w$ and black bud $b$ obtained by cutting in two halves an edge $e^*$ of the dual-constellation $C^*$ crossing $A$, and want to prove that the buds $w$ and $b$ will be glued together during the closure of $N$. We consider the set of buds encountered between $w$ and $b$ when turning clockwise around  $N$ starting from the bud $w$. Note that when $N$ and $A$ are superimposed, turning around the face of $N$ is the same as turning (counterclockwise) around the arborescence $A$. Let $e$ be the edge of $A$ crossing the edge $e^*$, and let $A'$ be the subtree of $A-\{e\}$ not containing the vertex $v_0$.   The buds between $w$ and $b$ around $N$ are the buds cut by the subtree $A'$ (since the tree $A$ is oriented toward $v_0$).  We now reason by induction on the number of edges in $A'$ to show that $w$ and $b$ will be glued together during the closure of $N$. First observe that if $A'$ has no edge, then $w$ and $b$ are matching buds, hence they will indeed be glued together during the closure of $N$.  Now if $A'$ has some edges, we know by induction that all the buds between $w$ and $b$ will be glued in pairs during the closure, therefore $w$ and $b$ will eventually be matching, hence will be glued together during the closure of $N$.  Thus, the dual-closure of $N$ is the tree-pointed constellation $(C,A)$. In other words, $\Delta\circ\Lambda(C,A)=(C,A)$. 

We have proved that $\Lambda$ and $\Delta$ are inverse mappings, so the dual-opening $\Lambda$ is a bijection. This completes the proof of Theorem~\ref{thm:bij-nebulas}.\hfill $\square$\\




\section{From nebulas to biddings}\label{sec:biddings}
In this section we draw a connection between nebulas and the set $\mM^n_{p_1,\ldots,p_k}$ whose cardinality appears in Theorem~\ref{thm:Jackson-counting}. More precisely we will encode nebulas by \emph{biddings}, where a \emph{bidding} of size $n$ and type $(p_1,\ldots,p_k)$ is a pair $((\om_1,\ldots,\om_k),(R_1,\ldots,R_n))$, where $\om_1,\ldots,\om_k$ are permutations of $[n]$, and $(R_1,\ldots,R_n)$ belongs to $\mM^n_{p_1,\ldots,p_k}$. This will lead to the probabilistic puzzle stated in Theorem~\ref{thm:tree-puzzle}.


A nebula of size $n$ and type $(p_1,\ldots,p_k)$ is said to be \emph{labelled} if its $n$ black vertices are given distinct labels in $[n]$, and the $p_t$ white buds of type $t$ are given distinct labels in $L_t$, where $L_t$ is the set of labels of the $p_t$ black vertices incident to the black buds of type $t$. Clearly there are $n!\prod_t p_t!$ ways of labelling a nebula of size $n$ and type $(p_1,\ldots,p_k)$. We also define the label of any edge of the nebula to be the label of the incident black vertex, so that for every pair $(t,i)\in[k]\times[n]$ there is either an edge or a white bud of type $t$ and label $i$. We denote by $c(t,i)$ the corner preceding either the edge or white bud of type $t$ labelled $i$ in clockwise order around the incident white vertex.

Let $N$ be a labelled rooted $k$-nebula of size $n$. We will now consider the sequence of \emph{white corners} (corners incident to white vertices) encountered when turning around the nebula $N$. Recall that \emph{turning clockwise around $N$} means walking around the face of $N$ by following the edges of $N$, with the edges on the right-side of the walker (the buds are just \emph{crossed}). A clockwise tour of a nebula is indicated in Figure~\ref{fig:bidding}. By turning clockwise around $N$ each of the white corners $\{c(t,i),~t\in[k],i\in[n]\}$ are visited, and this defines a cyclic order on $[k]\times[n]$. This cyclic order then gives a total order on the set $[k]\times[n]$ by choosing $(k,\ell_0)$ to be the greatest element, where $\ell_0$ is the label of the (black) root vertex of $N$. We call this total order the \emph{appearance order} on $[k]\times[n]$ and we denote it by $\prec_N$,
\begin{definition}
For a rooted nebula $N$, we denote by $\Psi(N)$ the bidding $((\om_1,\ldots,\om_k),(R_1,\ldots,R_n))$ defined as follows:
\begin{compactitem}
\item[(i)] for all $i\in[n]$, $R_i\subsetneq [k]$ is the set of types of the buds incident to the black vertex labelled $i$,
\item[(ii)] for all $t\in[k]$, $\om_t$ is the permutation of $n$ giving the appearance order of the edges or white buds of type $t$, that is, $(t,\om_t(1))\prec_N(t,\om_t(2))\prec_N\cdots \prec_N(t,\om_t(n))$.
\end{compactitem}
\end{definition}

\begin{figure}[h] 
\begin{center} 
\includegraphics[width=\linewidth]{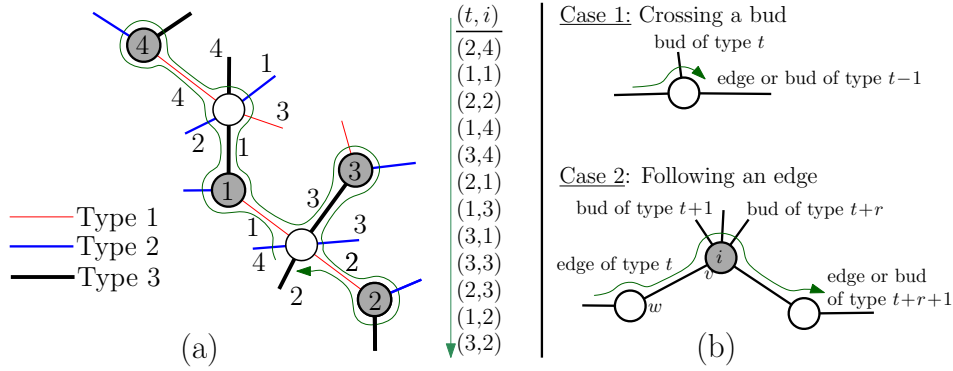} 
\caption{(a) The clockwise tour of a nebula $N$. Here $k=3$, $n=4$ and the root vertex has label $\ell_0=2$. The appearance order on $[k]\times[n]$ is indicated vertically (from top to bottom). The bidding $\Psi(N)=((\om_1,\om_2,\om_3),(R_1,R_2,R_3,R_4))$ is given by $R_1=\{2\}$, $R_2=\{2,3\}$, $R_3=\{1,2\}$, $R_4=\{2,3\}$ and $\om_1=1432$, $\om_2=3214$, $\om_3=4132$ (the permutations are indicated in one-line notation here). (b) From a white corner to the next during the clockwise tour of a nebula.} \label{fig:bidding}
\end{center} 
\end{figure} 

The mapping $\Psi$ is indicated in Figure~\ref{fig:bidding}. Recall now Definition~\ref{def:alpha} of the mapping $\al$. We say that a bidding $((\om_1,\ldots,\om_k),(R_1,\ldots,R_n))$ is \emph{valid} if the graph $\al((\om_1(n),\ldots,\om_{k-1}(n)),(R_1,\ldots,R_n))$ is a tree. We now state the main result of this section.

\begin{theorem}\label{thm:bij-bidding}
The mapping $\Psi$ is a bijection between labelled rooted nebulas of size $n$ and type $(p_1,\ldots,p_k)$ and valid biddings of size $n$ and type $(p_1,\ldots,p_k)$.
\end{theorem}

\begin{corollary} \label{cor:bij-bidding}
The valid biddings of size $n$ and type $(p_1,\ldots,p_k)$ are in $n!$-to-1 correspondence with the disjoint union
$\mC^n_{p_1+1,p_2,\ldots,p_k } \cup \cdots \cup \mC^n_{p_1,p_2,\ldots,p_k+1}$ of colored cacti. Consequently, Jackson's counting formula (Theorem~\ref{thm:Jackson-counting}) is equivalent to Theorem~\ref{thm:tree-puzzle}. 
\end{corollary}

\begin{proof}[Proof of Corollary~\ref{cor:bij-bidding}] 
Recall that there are $n!\prod_t p_t!$ ways of labelling a rooted nebulas of size $n$ and type $(p_1,\ldots,p_k)$. Hence, by Corollary~\ref{cor:bij-nebulas}, labelled rooted nebulas of size $n$ and type $(p_1,\ldots,p_k)$ are in  $n!$-to-1 correspondence with the union $\mC^n_{p_1+1,p_2,\ldots,p_k } \cup \cdots \cup \mC^n_{p_1,p_2,\ldots,p_k+1}$ of colored cacti. Thus the first statement of Corollary~\ref{cor:bij-bidding} is a direct consequence of Theorem~\ref{thm:bij-bidding}. 

We now prove the second statement of Corollary \ref{cor:bij-bidding}. First observe that Theorem~\ref{thm:Jackson-counting} implies that the set $\widetilde{\mC}^n_{p_1,p_2,\ldots,p_k }:=\mC^n_{p_1+1,p_2,\ldots,p_k } \cup  \cdots \cup \mC^n_{p_1,p_2,\ldots,p_k+1}$ has cardinality 
\begin{eqnarray}\nonumber
|\widetilde{\mC}^n_{p_1,p_2,\ldots,p_k }|={n!}^{k-1}\left(M^{n-1}_{p_1,p_2-1,\ldots,p_k-1}+\cdots +M^{n-1}_{p_1-1,\ldots,p_{k-1}-1,p_k}\right).
\end{eqnarray}
Conversely it is easy to see that this equation implies Theorem~\ref{thm:Jackson-counting} (that is, $|\mC^n_{p_1,p_2,\ldots,p_k }|={n!}^{k-1}M^{n-1}_{p_1-1,p_2-1,\ldots,p_k-1}$) by induction on $p_1$ starting with the base case 
\begin{equation}\label{eq:base-case}
|\mC^n_{1,p_2,\ldots,p_k}|={n!}^{k-1}\prod_{t=2}^k{n-1\choose p_t-1}={n!}^{k-1}\,M^{n-1}_{0,p_2-1,\ldots,p_k-1}.
\end{equation}
This base case can be checked as follows. The first equality in \eqref{eq:base-case} holds because there are $n!{n-1\choose p-1}$ permutations with cycles colored using every color in $[p]$, and for any permutations $\pi_2,\ldots,\pi_k$ there is a unique permutation $\pi_1$ such that $\pi_1\circ \pi_2\circ \cdots \circ\pi_k=(1,2,\ldots,n)$ and a unique coloring of the cycles of this permutation $\pi_1$ with 1 color. The second equality in \eqref{eq:base-case} holds because $M^{n-1}_{0,p_2-1,\ldots,p_k-1}$ is the set of $(n-1)$-tuples of subsets of $\{2,\ldots,k\}$ for which every $t\in\{2,\ldots,k\}$ appears in exactly $p_t-1$ subsets (so that there are ${n-1\choose p_t-1}$ ways of choosing in which subsets $t$ appears).

Observe now that $M^{n-1}_{p_1,p_2-1,\ldots,p_k-1}+\cdots +M^{n-1}_{p_1-1,\ldots,p_{k-1}-1,p_k}$ can be interpreted as the cardinality of the set $\widetilde{\mM}^n_{p_1,p_2,\ldots,p_k }$ of tuples $(R_1,\ldots,R_n)$ in $\mM^n_{p_1,p_2,\ldots,p_k}$ such that $R_1$ has cardinality $k-1$ (indeed $M^{n-1}_{p_1,p_2-1,\ldots,p_k-1}$ counts the tuples $(R_1,\ldots,R_n)$ in $\mM^n_{p_1,p_2,\ldots,p_k}$ such that $R_1=\{2,3,\ldots,k\}$ etc.). Therefore the preceding discussion shows that Theorem~\ref{thm:Jackson-counting} is equivalent to 
\begin{eqnarray}\label{eq:induction-Jackson}
n!\, |\widetilde{\mC}^n_{p_1,p_2,\ldots,p_k }|={n!}^{k}\,|\widetilde{\mM}^n_{p_1,p_2,\ldots,p_k }|.
\end{eqnarray}
Now, by Theorem~\ref{thm:bij-bidding}, the left-hand side of \eqref{eq:induction-Jackson} is the number of valid biddings of size $n$ and type $(p_1,\ldots,p_k)$, while the right-hand side is  the number of biddings of size $n$ and type $(p_1,\ldots,p_k)$ such that $R_1$ has cardinality $k-1$. This shows that Theorem~\ref{thm:Jackson-counting} is equivalent to Theorem~\ref{thm:tree-puzzle}. 
\end{proof}

The remaining of this section is devoted to the proof of Theorem~\ref{thm:bij-bidding}. In order to analyze the mapping $\Psi$ we introduce an intermediate class of objects called \emph{prebidding} and define a mapping $\vartheta$ between labelled nebulas and \emph{prebiddings}, and then a mapping $\sigma$ between prebiddings and biddings such that $\Psi=\sigma\circ\vartheta$.

A \emph{prebidding} is a pair $(\prec, (R_1,\ldots,R_n))$ where $\prec$ is a linear order on $[k]\times [n]$ and $R_1,\ldots,R_n$ are strict subsets of $[k]$. We say that a prebidding is \emph{valid} if the greatest element is of the form $(k,\ell_0)$ for some $\ell_0\in[n]$ and if, whenever $(t,i)$ and $(t',i')$ are consecutive pairs in the order $\prec$, or when $(t,i)=(k,\ell_0)$ and $(t',i')$ is the least element, one has $t'=\al(t,R_i)$ (that is, $t'=t-1$ if $t\in R_i$ and $t'=t+r$ if $t\notin R_i$, $t+1,\ldots,t+r\in R_i$ and $t+r+1\notin R_i$). 

We now consider the mapping $\vartheta$ which associates to a labelled rooted $k$-nebula $N$ the prebidding $\vartheta(N)=(\prec_N, (R_1,\ldots,R_n))$, where $\prec_N$ is the appearance order on $[k]\times[n]$ and $R_i\subsetneq [k]$ is the set of types of the buds incident to the black vertex labelled $i$.

\begin{lemma}\label{lem:vartheta}
The mapping $\vartheta$ is a bijection between labelled rooted nebulas and valid prebiddings.
\end{lemma}

\begin{proof}
Let $N$ be a nebula. We first prove that the prebidding $\vartheta(N)=(\prec_N, (R_1,\ldots,R_n))$ is valid. It is clear from the definition of the appearance order $\prec_N$ that the greatest element is of the form $(k,\ell_0)$ for some $\ell_0\in[n]$. We now consider pairs $(t,i)$ and $(t',i')$ which are either consecutive in the appearance order or such that $(t,i)=(k,\ell_0)$ and $(t',i')$ is the least element of the appearance order. By definition, if $t\in R_i$ the pair $(t,i)$ corresponds to a white bud of $N$ and if $t\notin R_i$ it corresponds to an edge of $N$. 
We first suppose that $t\in R_i$. Since $t\in R_i$, the white corner $c(t,i)$ precedes the white bud $b$ of type $t$ labelled $i$, so that after visiting the corner $c(t,i)$ the clockwise tour of $N$ crosses the bud $b$ and arrives at a white corner preceding either a bud or an edge of type $t-1$ (see Figure~\ref{fig:bidding}(b)). Thus $t'=t-1$.  
Suppose now that $t\notin R_i$. In this case the pair $(t,i)$ corresponds to an edge $e$ incident to the black vertex $v$ labelled $i$ and to a white vertex $w$. After passing through the white corner $c(t,i)$ preceding $e$ around $w$, the clockwise tour follows $e$, crosses the buds of type $t+1,\ldots,t+r\in R_i$ around the black vertex $v$, then follows an edge of type $t+r+1\notin R_i$ and arrives at a white corner preceding either an edge or a bud of type $t+r$ (see Figure~\ref{fig:bidding}(b)). Thus $t'=t+r$. This completes the proof that the prebidding $\vartheta(N)$ is valid. 

We now prove the injectivity of the mapping $\vartheta$. Observe that from the prebidding $\vartheta(N)$ one can deduce the complete list of edges followed and buds crossed during the clockwise tour of the nebula $N$. This determines a polygon with black and white vertices, and with buds drawn on the inside region. This is illustrated in Figure \ref{fig:bidding-to-nebulas}. Moreover the edges come in pairs having the same types and labels so that one knows how to glue the edges of the polygon in pairs in order to recover the nebula $N$. Lastly the root vertex of $N$ is identified as the black vertex of label $\ell_0$, where $(k,\ell_0)$ is the greatest element of $\prec_N$. 

In order to prove the surjectivity of $\vartheta$ we consider a valid prebidding $P=(\prec, (R_1,\ldots,R_n))$ and want to check that one can apply the above mentioned procedure to get a nebula (see Figure \ref{fig:bidding-to-nebulas}). From the prebidding $P$ one can construct a polygon with black and white vertices, and with buds drawn on the inside region such that edges and buds have types which increase clockwise around white vertices and decrease clockwise around white vertices (the polygon is constructed in such a way that its clockwise tour is described by the prebidding $P$). Now it is easy to check that for each type $t$ and label $i$ either there is both a white bud and a black bud of type $t$ labelled $i$ (this happens if $t\in R_i$) or there are two edges of type $t$ labelled $i$ one going from a black vertex to a white vertex, and one going the opposite way when turning clockwise around the polygon (this happens if $t\notin R_i$). Thus one can glue together the pairs of edges of the same types and labels, and thereby obtain a map $N$ with a single face such that its clockwise tour is described by the prebidding $P$. In the map $N$ the black vertices have degree $k$ since for every $i\in[n]$ the black vertex labelled $i$ is incident exactly to the edges and black buds labelled $i$. Thus $N$ is a nebula. Lastly it is clear that $\vartheta(N)=P$. This proves the surjectivity and hence the bijectivity of $\vartheta$.
\end{proof}

\begin{figure}[h] 
\begin{center} 
\includegraphics[width=\linewidth]{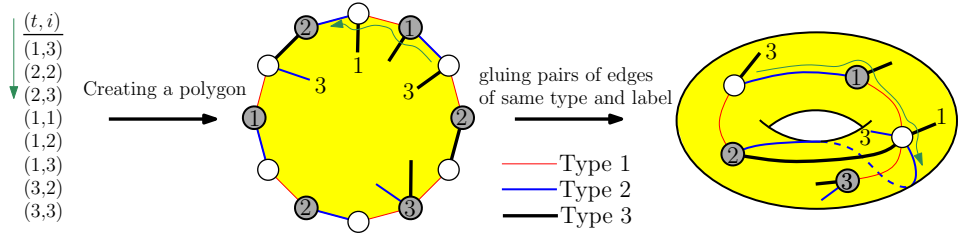} 
\caption{From valid prebiddings to nebulas (the appearance order is indicated from top to bottom).} \label{fig:bidding-to-nebulas}
\end{center} 
\end{figure} 

We now define a mapping $\sigma$ on valid prebiddings by setting $$\sigma(\prec, (R_1,\ldots,R_n))=((\om_1,\ldots,\om_k), (R_1,\ldots,R_n)),$$ where $\om_t$ is the permutation of $[n]$ defined by $(t,\om_t(1))\prec(t,\om_t(2))\prec\cdots \prec(t,\om_t(n))$. 
Observe that $\Psi=\sigma\circ\vartheta$. The proof of the following lemma is based on the BEST Theorem.
\begin{lemma}\label{lem:sigma}
The mapping $\sigma$ is a bijection between valid prebiddings and valid biddings.
\end{lemma}

\begin{proof}
Let $(R_1,\ldots,R_n)$ be a fixed tuple of strict subsets of $[k]$.
We consider the directed graph $G$ with vertex set $[k]$ and arc set $\{a_{t,i},~(t,i)\in[k]\times[n]\}$, where $a_{t,i}$ is an arc from vertex $t$ to vertex $\al(t,R_i)$. First observe that every vertex $t$ of the graph $G$ has $n$ outgoing arcs (the arcs $a_{t,i}$, $i\in[n]$) and $n$ ingoing arcs (because for each $i\in[n]$ there a unique $t'\in [k]$ such that $\al(t',R_i)=t$). Now, to any Eulerian tour of $G$ one can associate a linear order $\prec$ on $[k]\times[n]$ defined by setting $(t,i)\prec (t',i')$ if the arc $a_{t,i}$ is taken before the arc $a_{t',i'}$ during the Eulerian tour. It is easily seen that this gives a bijection $\sigma_1$ between the linear orders $\prec$ such that $(\prec, (R_1,\ldots,R_n))$ is a valid prebidding and the Eulerian tours of $G$ starting (and ending) at vertex~$k$. Moreover, by the BEST Theorem, the Eulerian tours of $G$ starting at vertex~$k$ are in bijection with the assignments for each vertex $t\in[k]$ of a linear order of the arcs going out of $t$ (the order in which these arcs are used during the Eulerian tour) such that the set of greatest arcs going out of the vertices $1,2,\ldots,k-1$ form a spanning tree of $G$ directed toward $k$. Note that the linear order on the outgoing edges of a vertex $t$ can be represented by a permutation $\om_t$ defined by setting $\om_{t}(i)=j$ if $a_{t,j}$ is the arc going out of vertex $t$ used at the $i$th exit of that vertex during the Eulerian tour. Hence the BEST theorem gives a bijection $\sigma_2$ between the Eulerian tours of $G$ and the tuples of permutations $(\om_1,\ldots,\om_k)$ such that $((\om_1,\ldots,\om_k), (R_1,\ldots,R_n))$ is a valid biding. This completes the proof since $\sigma(\prec, (R_1,\ldots,R_n))=(\sigma_2\circ\sigma_1(\prec), (R_1,\ldots,R_n))$.
\end{proof}

By definition $\Psi=\sigma\circ\vartheta$, so Theorem~\ref{thm:bij-bidding} is a direct consequence of Lemmas~\ref{lem:vartheta} and~\ref{lem:sigma}.\hfill $\square$

\section{Cases $k=2,3,4$ of Theorem~\ref{thm:tree-puzzle}.}\label{sec:smallk}
We have shown (Corollary~\ref{cor:bij-bidding}) that for each $k\geq 2$ Jackson's counting formula is equivalent to Theorem~\ref{thm:tree-puzzle}.
In this section we prove a few cases of Theorem~\ref{thm:tree-puzzle}.\\

\noindent \textbf{Case $k=2$}. Let $B=((i),(R_1,\ldots, R_n))$, where $i\in[n]$ and $(R_1,\ldots, R_n)\in \mM^n_{p_1,p_2}$. The graph $G=\al(B)$ has vertex set $\{1,2\}$ and one arc $a_1=(1,x)$ where $x=2$ if $R_{i}=\{1\}$ or $R_{i}=\{2\}$, and $x=1$ otherwise (if $R_i$ is empty). Thus, $G$ is a tree if and only if $R_{i}$ has cardinality 1. Clearly the probability that $R_{i}$ has cardinality 1 is the same as the probability that $R_1$ has cardinality 1, thus Theorem~\ref{thm:tree-puzzle} is proved for $k=2$.\\
 
\noindent \textbf{Case $k=3$}. Let $B=((i,j),(R_1,\ldots, R_n))$, where $i,j\in[n]$ and $(R_1,\ldots, R_n)\in \mM^n_{p_1,p_2,p_3}$. The graph $G=\al(B)$ has vertex set $\{1,2,3\}$ and two arcs $a_1=(1,x)$, and $a_2=(2,y)$. We want to evaluate the probability that the graph $G$ is a tree. There are three possible trees and they occur in the events described by Figure~\ref{fig:casek3}. Using inclusion-exclusion one gets that the probability that $G$ is a tree is 
\begin{eqnarray}
\label{eq:k3brut}
\Pr(tree)&\!\!=\!\!&P_{\{1\}\!\times\!\{2\}}+P_{\{1\}\!\times\!\{3\}}+P_{\{2\}\!\times\!\{3\}} \nonumber\\
&&-P_{\{1\}\!\times\!\{2,3\}}-P_{\{2\}\!\times\!\{1,3\}}-P_{\{3\}\!\times\!\{1,2\}},\\
&&+P_{\{1,2\}\!\times\!\{1,3\}}+P_{\{1,2\}\!\times\!\{2,3\}}+P_{\{1,3\}\!\times\!\{2,3\}},\nonumber 
\end{eqnarray}
where $P_{{A}\times{B}}$ denotes the probability of the event $\{A\subseteq R_i \textrm{ and } B \subseteq R_j\}$ (all these probabilities depend on $n,p_1,p_2,p_3$). We then use the following easy lemma.
\begin{lemma}[Exchange Lemma for $k=3$]\label{lem:exchange}
For $\{a,b,c\}=\{1,2,3\}$ one has 
$$P_{\{a,b\}\times\emptyset}=P_{\{a\}\times\{b\}}-P_{\{a,c\}\times\{b\}}+P_{\{a,b\}\times\{a,c\}}.$$
\end{lemma}
\begin{proof}
We consider the events $E_1=\{a\in R_i,~c\notin R_i \textrm{ and } b\in R_j\}$ and $E_2=\{a,b\in R_i,~c\notin R_i \textrm{ and } R_j\neq\{a,c\}\}$. The probabilities of these event are $\Pr(E_1)=P_{\{a\}\times\{b\}}-P_{\{a,c\}\times\{b\}}$ and $\Pr(E_2)=P_{\{a,b\}\times\emptyset}-P_{\{a,b\}\times\{a,c\}}$. We now define a bijection between the events $E_1$ and $E_2$ by \emph{exchanging the $b$ content} of $R_i$ and $R_j$, that is, by changing the sets $R_i$ and $R_j$ into $R_i'$ and $R_j'$ respectively, where $R_i'=R_i\cup \{b\}$ if $b\in R_j$ and $R_i'=R_i\setminus \{b\}$ otherwise, and similarly $R_j'=R_j\cup \{b\}$ if $b\in R_i$ and $R_j'=R_j\setminus \{b\}$ otherwise. Since exchanging the $b$ content creates a bijection between the events $E_1$ and $E_2$ (which preserves the fact that $(R_1,\ldots,R_n)$ is in $\mM^n_{p_1,p_2,p_3}$) the probabilities of the events $E_1$ and $E_2$ are equal.
\end{proof}
\begin{figure}[h] 
\begin{center} 
\includegraphics[width=\linewidth]{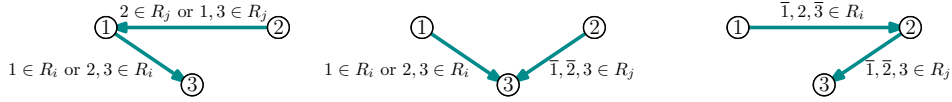} 
\caption{Events for which $G=\al(B)$ is a tree in the case $k=3$ of Theorem~\ref{thm:tree-puzzle}. The notation $\overline{a}\in R$ means $a\notin R$.} \label{fig:casek3}
\end{center} 
\end{figure} 

Using Lemma~\ref{lem:exchange} (three times) in \eqref{eq:k3brut} gives 
$$\Pr(tree)=P_{\{1,2\}\times\emptyset}+ P_{\{1,3\}\times\emptyset}+P_{\{2,3\}\times\emptyset}.$$ 
Thus the probability that $G=\al(B)$ is a tree is equal to the probability that $R_i$ has cardinality 2. Theorem~\ref{thm:tree-puzzle} follows.\\

\noindent \textbf{Case $k=4$ and greater}. In the case $k=4$ one considers $B=((i,j,\ell),(R_1,\ldots, R_n))$, where $i,j,\ell\in[n]$ and $(R_1,\ldots, R_n)\in \mM^n_{p_1,p_2,p_3,p_4}$. One gets the following analogue of \eqref{eq:k3brut}:
\begin{eqnarray}
\label{eq:k4brut}
\Pr(tree)&\!\!\!=\!\!\!& 
\widetilde{P}_{\{1\}\!\times\!\{2\}\!\times\!\{3\}}
-\widetilde{P}_{\{1\}\!\times\!\{2\}\!\times\!\{3,4\}}
+\widetilde{P}_{\{1\}\!\times\!\{2,3\}\!\times\!\{3,4\}}
-\widetilde{P}_{\{1\}\!\times\!\{2,3,4\}\!\times\!\{3\}}\nonumber\\
&&-\widetilde{P}_{\{1,2\}\!\times\!\{2,3\}\!\times\!\{3,4\}}
+\widetilde{P}_{\{1\}\!\times\!\{4,1,2\}\!\times\!\{2,3\}}
+\widetilde{P}_{\{1\}\!\times\!\{4,1,2\}\!\times\!\{3,4\}}
-\widetilde{P}_{\{1\}\!\times\!\{4,1,2\}\!\times\!\{2,3,4\}}\nonumber\\
&& +\widetilde{P}_{\{1,2,3\}\!\times\!\{3,4\}\!\times\!\{4,1\}}
-\widetilde{P}_{\{1,2\}\!\times\!\{2,3,4\}\!\times\!\{3,4,1\}}
+\widetilde{P}_{\{1,2,3\}\!\times\!\{3,4,1\}\!\times\!\{4,1,2\}}\nonumber
\end{eqnarray}
with $\tilde{P}_{A,B,C}=\sum_{t=1}^4P_{t+A,t+B,t+C}$, where $t+A$ is the cyclically shifted set $\{t+a,a\in A\}$ and $P_{A,B,C}$ is the probability of the event $\{A\subseteq R_i,~ B\subseteq R_j\textrm{ and } C\subseteq R_\ell\}$. This expression can be obtained by direct inspection of the 16 possible Cayley trees, or by using the matrix-tree theorem. Using the later method we were able to compute analogues of \eqref{eq:k3brut} for all $k\leq 9$. 

In the case $k=4$ we can use \emph{exchange lemmas} similar to Lemma~\ref{lem:exchange} to prove Theorem~\ref{lem:exchange} (one has to use six such lemmas, and the calculations can be organized according to the order of magnitude of these probabilities when $n$ tends to infinity, with $p_1,p_2,p_3,p_4$ kept fixed). One can in fact generalize the type of exchange lemmas used, and use Gr\"obner bases to prove Theorem~\ref{thm:tree-puzzle} (hence Theorem~\ref{thm:Jackson-counting}) up to $k=7$. There is however one extra difficulty starting with $k=4$: the \emph{exchange operation} as the one used in the Proof of the lemma~\ref{lem:exchange} become invalid unless one conditions on the fact that the indices $i,j,\ell,\ldots$ under consideration are distinct (otherwise the exchange operation might affect one of the sets outside of the exchange). Thus for $k\geq 4$, the computer assisted proofs require to consider the different subcases corresponding to having non-distinct indices among $i,j,\ell,\ldots$. In a forthcoming paper \cite{BM:trees-from-sets} we shall give a general proof of Theorem~\ref{thm:tree-puzzle} for all $k$ and prove other similar results.

\bibliographystyle{plain} 
\bibliography{biblio-constellations} 
\label{sec:biblio} 
 
\small

\noindent Olivier Bernardi, Alejandro H. Morales, 
Department of Mathematics, Massachusetts Institute of Technology, 
Cambridge, MA USA 02139.\\
\{{\tt bernardi, ahmorales}\}{\tt @math.mit.edu}

\end{document}